\documentclass{article}

\usepackage{amsmath}
\usepackage{amssymb}
\usepackage{amsfonts}
\usepackage{amsthm}
\usepackage[margin=1in]{geometry}
\usepackage{hyperref}
\usepackage[width=14cm,format=plain,labelfont=sf,bf,textfont=sf,up,justification=justified,labelsep=period]{caption}
\usepackage{graphicx}
\usepackage{framed}
\usepackage{xcolor}
\usepackage{tabularx}
\usepackage[mathscr]{euscript} 
\usepackage{amsbsy} 

\hypersetup{
    colorlinks=true,       
    linkcolor=blue,          
    citecolor=blue,        
    filecolor=blue,      
    urlcolor=blue           
}

\makeatletter
\def\iddots{\mathinner{\mkern1mu\raise\p@
\vbox{\kern7\p@\hbox{.}}\mkern2mu
\raise4\p@\hbox{.}\mkern2mu\raise7\p@\hbox{.}\mkern1mu}}
\makeatother

\newcommand{\field}[1]{\mathbb{#1}}

\newcommand{\cA}{\mathcal A}

\newcommand{\cK}{\mathcal K}

\newcommand{\cU}{\mathcal U}
\newcommand{\cV}{\mathcal V}

\newcommand{\NN}{\field{N}}
\newcommand{\RR}{\field{R}}

\DeclareMathOperator{\rank}{rank}

\DeclareMathOperator{\rankpsd}{\rank_{psd}}

\DeclareMathOperator{\conv}{conv}
\DeclareMathOperator{\diag}{diag}

\DeclareMathOperator{\xc}{xc} 
\DeclareMathOperator{\1}{\mathbf{1}}
\DeclareMathOperator{\COR}{COR} 
\DeclareMathOperator{\UDISJ}{UDISJ} 
\DeclareMathOperator{\val}{val} 

\let\Im\relax 
\DeclareMathOperator{\Im}{Im}
\DeclareMathOperator{\Ker}{Ker}

\newcommand{\bM}{{\mathbf M}}

\newcommand{\ra}{\texttt{a}}
\newcommand{\rb}{\texttt{b}}
\newcommand{\rc}{\texttt{c}}
\newcommand{\rd}{\texttt{d}}

\renewcommand{\S}{{\bf S}}
\renewcommand{\L}{{\bf L}} 

\newcommand{\C}{{\mathcal C}}

\newcommand{\concat}{\cdot} 

{\begin{small}\begin{list}{}%
         {\setlength{\leftmargin}{1cm}\setlength{\rightmargin}{1cm}}%
         \item[]%
}
{\end{list}\end{small}}

\theoremstyle{plain}
\newtheorem{prop}{Proposition}
\newtheorem{lem}{Lemma}
\newtheorem{thm}{Theorem}

\newtheorem{cor}{Corollary}

\theoremstyle{definition}
\newtheorem{defn}{Definition}

\theoremstyle{remark}
\newtheorem{example}{Example}
\newtheorem*{rem}{Remark}

\title{
\vspace{-1.1cm}
Exponential lower bounds on fixed-size psd rank\\ and semidefinite extension complexity}

\author{Hamza Fawzi \and Pablo A. Parrilo \thanks{The authors are with
    the Laboratory for Information and Decision Systems, Department of
    Electrical Engineering and Computer Science, Massachusetts
    Institute of Technology, Cambridge, MA 02139. Email:
    \texttt{\{hfawzi,parrilo\}@mit.edu}.}}
\renewcommand\footnotemark{}

\date{November 11, 2013}

\begin{document}
\maketitle
\vspace{-0.5cm}
\begin{abstract}
There has been a lot of interest recently in proving lower bounds on the size of linear programs needed to represent a given polytope $P$. In a breakthrough paper Fiorini et al. \cite{fiorini2012linear} showed that any linear programming formulation of maximum-cut must have exponential size. A natural question to ask is whether one can prove such strong lower bounds for \emph{semidefinite programming} formulations. In this paper we take a step towards this goal and we prove strong lower bounds for a certain class of SDP formulations, namely SDPs over the Cartesian product cone $\S^d_+\times \dots \times \S^d_+ = (\S^d_+)^r$ when $d$ is constant ($\S^d_+$ is the cone of $d\times d$ positive semidefinite matrices). In practice this corresponds to semidefinite programs with a block-diagonal structure and where blocks have size $d$. We show that any such extended formulation of the cut polytope must have exponential size when the size of the block $d$ is a fixed constant.
The result of Fiorini et al. for LP formulations is obtained as a special case when $d=1$. For blocks of size $d=2$ the result rules out any small formulations using \emph{second-order cone programming}. Our study of SDP lifts over Cartesian product $(\S^d_+)^r$ is motivated mainly from practical considerations where it is well known that such SDPs can be solved more efficiently than general SDPs. The proof of our lower bound relies on new results about the sparsity pattern of certain matrices with small psd rank, combined with an induction argument inspired from the recent paper by Kaibel and Weltge \cite{kaibel2013short} on the LP extension complexity of the correlation polytope.

\end{abstract}


\section{Introduction}

\subsection{Preliminaries}

Linear programming and semidefinite programming play a crucial role in the design of algorithms \cite{williamson2011design}. A line of work initiated by Yannakakis in \cite{yannakakis1991expressing} and which has received a lot of attention recently studies the limits of linear programming and their expressive power. The main object of study there is the notion of \emph{extended formulation} of a polytope $P$.
Given a polytope $P \subset \RR^n$, an \emph{extended formulation} of $P$ is a representation of $P$ as the projection of a higher-dimensional polytope $Q \subset \RR^m$ with $m\geq n$, i.e., $P = \pi(Q)$ where $\pi$ is a linear projection map $\pi:\RR^m\rightarrow \RR^n$. The size of an extended formulation is the number of facets of the polytope $Q$, i.e., the number of linear inequalities needed to describe $Q$. The \emph{extension complexity} of $P$, denoted $\xc(P)$, is defined as the size of the smallest extended formulation of $P$. There are many examples of polytopes $P$ where $\xc(P)$ is much smaller than the number of facets of $P$. For example the \emph{cross-polytope} (i.e., the unit $\ell_1$-ball) in $\RR^n$ has $2^n$ facets but has a simple extended formulation of size $2n$. Another example is the \emph{permutahedron} defined as the convex hull of all the permutations of $(1,2,\dots,n)$. Even though the permutahedron has $2^n-2$ facets one can show that its extension complexity is $\Theta(n\log n)$ (cf. \cite{goemans2009smallest}).

In a breakthrough paper \cite{fiorini2012linear}, Fiorini et al. proved that there is no polynomial-sized extended formulations of the \emph{cut polytope} of the complete graph $K_n$. The cut polytope is the natural polytope of interest for the maximum-cut problem and is defined as the convex hull of the incidence vectors of cuts in $K_n$ \cite{deza1997geometry}. Instead of working directly with the cut polytope, the authors \cite{fiorini2012linear} rather worked with a closely related polytope called the \emph{correlation polytope} $\COR(n)$ and which is defined as the convex hull of the outer products $bb^T$ for $b \in \{0,1\}^n$:
\[ \COR(n) = \conv\left( bb^T \; : \; b \in \{0,1\}^n \right). \]
One can show that the correlation polytope is linearly isomorphic to the cut polytope of the complete graph $K_{n+1}$ on $n+1$ vertices, see e.g., \cite{deza1997geometry}.
Fiorini et al. \cite{fiorini2012linear} proved that the extension complexity of the correlation polytope (and thus of the cut polytope) is exponentially large in $n$.
Their proof relied on a key lemma of Razborov \cite{razborov1992distributional} in the context of communication complexity. Subsequently, different proofs of this exponential lower bound were then given in \cite{braverman2013information} and in \cite{braun2013common} and in particular the recent paper by Kaibel and Weltge \cite{kaibel2013short} gives a short proof of this fact.


An important problem that remains open is to obtain strong lower bounds on the size of \emph{semidefinite programming} formulations.
Semidefinite programming has played a crucial role in the design of approximations algorithms and has been applied to many hard combinatorial problems, like e.g., max-cut, graph-coloring, etc. \cite{goemans1995improved,karger1998approximate,arora2009expander}. It has also been shown that, under the Unique Games Conjecture, semidefinite programs allow to get the best possible approximation ratio for a wide class of problems \cite{raghavendra2008optimal}.
A natural question to consider following the result of Fiorini et al. \cite{fiorini2012linear} is to prove strong lower bounds for SDP formulations of the cut polytope.
In this paper we take a step towards this goal and we prove strong lower bounds for a certain class of SDP extended formulations. We consider semidefinite programs over the Cartesian product cone $\S^d_+ \times \dots \times \S^d_+ = (\S^d_+)^r$ where $\S^d_+$ is the cone of $d\times d$ positive semidefinite matrices and where $d$ is a \emph{fixed constant}. Recall that the standard form of a semidefinite program is:
\begin{equation}
\label{eq:canonicalsdp}
\begin{array}{ll}
\text{minimize} & \langle C, X \rangle\\
\text{subject to} & A(X) = b\\
                  & X \in \S^N_+
\end{array}
\end{equation}
where $\S^N_+$ is the cone of $N\times N$ real symmetric matrices and $A$ is a linear map. In other words, the feasible set of a semidefinite program is the intersection of the cone $\S^N_+$ with an affine subspace $\{ X \in \S^N \; : \; A(X) = b \}$.
The canonical form \eqref{eq:canonicalsdp} of a semidefinite program is interesting from a theoretical point of view but it can sometimes hide structural information which can be crucial for solving the SDP in practice. In this work we are interested in semidefinite programs where the variable $X$ is constrained to be block-diagonal and where the diagonal blocks $X_1,X_2,\dots,X_r$ have fixed size $d$. Such semidefinite programs can be written as follows:
\begin{equation}
\label{eq:blockdiagonal-cartesiansdp}
\begin{array}{ll}
\text{minimize} & \langle C_1, X_1 \rangle + \dots + \langle C_k, X_r \rangle\\
\text{subject to} & A(X_1,\dots,X_r) = b\\
                  & (X_1,\dots,X_r) \in \underbrace{\S^{d}_+\times \dots \times \S^{d}_+}_{\text{(r copies)}}
\end{array}
\end{equation}
where $r$ is the number of blocks on the diagonal.
It is well known in numerical optimization that problems of the form \eqref{eq:blockdiagonal-cartesiansdp} can be solved a lot faster than general problems in the form \eqref{eq:canonicalsdp}. In other words, solving an SDP over the cone $\cK = \S^{d}_+ \times \dots \times \S^{d} = (\S^d_+)^r$ can be done more efficiently than solving a general SDP over the cone $\cK = \S^{N}_+$ where $N=dr$. Note that this phenomenon is proper to semidefinite programming and does not exist in linear programming since $\RR^{N}_+ = (\RR^{d}_+)^r$ when $N=dr$.

To prove lower bounds on LP extension complexity, the previously cited papers \cite{fiorini2012linear,braverman2013information,braun2013common} exploit the connection established by Yannakakis in \cite{yannakakis1991expressing} between \emph{extended formulations} and \emph{nonnegative rank}. This connection was generalized in \cite{gouveia2011lifts} to conic extended formulations and in particular semidefinite programming formulations. A key quantity identified in \cite{gouveia2011lifts} is the \emph{psd rank} and was shown to characterize the size of semidefinite programming formulations of polytopes. The \emph{psd rank} of a matrix has an interesting interpretation in quantum information theory in the context of correlation generation \cite{jain2013efficient}: Consider the problem of simulating a given conditional distribution $p(y|x)$ whereby the output $Y$ is obtained by measuring a quantum state $\omega_x$ through a POVM $\{F_y\}$. The \emph{psd rank} of the matrix $p(y|x)$ gives the smallest dimension of quantum states needed to simulate the conditional distribution $p(y|x)$. The block-diagonal requirement studied in this paper then has a natural interpretation in quantum information: it requires the states $\omega_x$ to be classical-quantum (also called \emph{cq-states}) where the dimension of the quantum part is constant.


\paragraph{Formal statement of result} We now formally state the main result of this paper. 
Given two integers $d$ and $r$, we say that a polytope $P$ has a $(\S^d_+)^r$-lift if $P$ is the projection of a feasible set of an SDP over $(\S^d_+)^r$, i.e., if
\[ P = \pi((\S^d_+)^r \cap L) \]
where $L$ is an affine subspace and $\pi$ a linear projection map. 

We show that for \emph{any constant} $d$, any $(\S^{d}_+)^r$-lift of the correlation polytope $\COR(n)$ must have exponential size, i.e., $r$ must be exponential in $n$. This is stated in the following theorem:
\begin{thm}[Main]
\label{thm:main}
Let $d \geq 1$ be a fixed integer. For $n \geq d$, if $\COR(n)$ has a $(\S^{d}_+)^r$-lift for some $r \in \NN$, then necessarily 
\[ r \geq \kappa(d) \cdot c(d)^n, \]
where $c(d) = (1-1/3^{d})^{-1/{d}} > 1$ is a constant greater than 1 and $\kappa(d) = (3^{d}-1)^{-(1-1/d)}$.

For the special case $d=2$, the constants $\kappa(2)$ and $c(2)$ can be taken to be:
\[ \kappa(2) = \frac{1}{\sqrt{7}} \qquad c(2) = \sqrt{\frac{9}{7}} \approx 1.13. \]
\end{thm}

As mentioned earlier, the polytope $\COR(n)$ is linearly isomorphic to the cut polytope. Fiorini et al. \cite[Lemma 11]{fiorini2012linear} also showed that the correlation polytope projects onto a face of the TSP polytope of size $O(n^2)$.
 The result above thus rules out any polynomial-size formulation of the cut polytope and the TSP polytope using block-diagonal semidefinite programs. 

Theorem \ref{thm:main} generalizes the existing lower bounds on LP extended formulations of the correlation polytope. Indeed LP extended formulations are captured by the case $d=1$ in the theorem above. The case $d=2$ is also important since it amounts to asking what is the smallest size of a \emph{second-order cone program} needed to represent $\COR(n)$. Recall that the second-order cone $\L^k$ in $\RR^{k+1}$ is defined by:
\[ \L^k = \{(t,x) \in \RR \times \RR^k \; : \; \|x\|_2 \leq t \}. \]
The cone $\S^2_+$ is affinely isomorphic to $\L^2$ since we have:
\[  t \geq \sqrt{x^2+y^2}  \quad \Leftrightarrow \quad \begin{bmatrix} t+x & y\\ y & t-x \end{bmatrix} \succeq 0. \]
Also it is known that for second-order cone programming, it is sufficient to work with Cartesian products of $\L^2$. Indeed, any constraint on the second-order cone $\L^k$ with $k\geq 2$ can be represented with $k-1$ constraints $\L^2$ and $k-2$ additional variables. For example, we have:
\[ 
(t,x_1,x_2,x_3,x_4) \in \L^4
\quad
\Leftrightarrow
\quad
\exists u, v \in \RR \text{ s.t. } 
\begin{cases}
(t,u,v) \in \L^2 \\
(u,x_1,x_2) \in \L^2\\
(v,x_3,x_4) \in \L^2.
\end{cases}
\]
For the general procedure to convert a constraint on $\L^k$ to a set of constraints on $\L^2$, we refer the reader to \cite[Section 2]{ben2001polyhedral}.

\paragraph{Strategy of proof} Our proof of Theorem \ref{thm:main} goes by analyzing the so-called \emph{unique-disjointness matrix} $\UDISJ(n)$ (cf. Definition \ref{def:udisj}) studied by Fiorini et al. in \cite{fiorini2012linear} to prove lower bounds on the LP extension complexity of the correlation polytope. In \cite{fiorini2012linear} and the subsequent papers \cite{braverman2013information,braun2013common,kaibel2013short} it was shown that the nonnegative rank of $\UDISJ(n)$ is exponentially large in $n$. To prove our lower bound of Theorem \ref{thm:main} concerning $(\S^d_+)^r$-lifts, we need to deal with a different notion of rank that can be regarded as an intermediate between the nonnegative rank and the psd rank: instead of requiring each term in the factorization of the matrix to be nonnegative and rank-one (as in the definition of the nonnegative rank), we ask instead that each term be nonnegative and has \emph{psd-rank} $\leq d$ (cf. Definition \ref{def:psdrank} for the definition of psd rank). To prove our theorem, we show that any such decomposition of the unique-disjointness matrix requires an exponential number of terms, when $d$ is fixed. Our approach proceeds by analyzing the set of matrices of psd rank $\leq  d$ which can arise in a nonnegative decomposition of $\UDISJ(n)$. This set of matrices can be complicated because of the psd rank constraint and little is currently known about how to obtain good bounds on the psd rank \cite{lee2012support}. Our proof relies on new results about the sparsity pattern of certain matrices with small psd rank, which we combine with an induction argument inspired by the recent paper of Kaibel and Weltge  \cite{kaibel2013short}. We identify a new property called the \emph{uniform covering} property which is essential to our induction argument and which generalizes a simple observation about the sparsity pattern of certain $2\times 2$ matrices of rank one. The proof of our main theorem is composed of two parts: in the first part we show how
to use the \emph{uniform covering} property to obtain strong lower bounds; and in the second part we construct good \emph{uniform coverings} for the class of matrices of interest.

\paragraph{Organization} The rest of the paper is devoted to the proof of Theorem \ref{thm:main}. We start by reviewing the main definitions and theorems concerning conic extended formulations and the slack matrix of a polytope \cite{gouveia2011lifts}. After briefly revisiting the induction argument of Kaibel and Weltge \cite{kaibel2013short}, we introduce the notion of \emph{uniform covering} and we show how uniform coverings can be used to prove lower bounds using induction. In Sections \ref{sec:uniformcovering-d=2} and \ref{sec:uniformcovering-general} we prove the existence of uniform coverings for the matrices of interest, first for the case $d=2$ and then for the general case. Then we show how to combine these results together to obtain the exponential lower bound of Theorem \ref{thm:main}.

\section{Proof of main result}

\subsection{Definitions}
\label{sec:defs}

\paragraph{Lifts of polytopes and slack matrix} We first recall some basic definitions concerning lifts of polytopes and the notion of slack matrix. Let $P \subset \RR^n$ be a polytope. If $\cK$ is a convex cone in $\RR^m$ (with $m \geq n$), we say that $P$ has a $\cK$-lift \cite{gouveia2011lifts} if there exists an affine subspace $L$ of $\RR^m$ and a linear map $\pi:\RR^m \rightarrow \RR^n$ such that $P = \pi(\cK \cap L)$.

Let $v_1,\dots,v_{V}$ be the vertices of $P$ and let $\langle a_i, x \rangle \leq b_i$, $i=1,\dots,F$ be the facet-defining inequalities of $P$. The slack matrix of $P$ is a matrix $S \in \RR^{F\times V}$ where $S_{i,j}$ is the slack of the $j$'th vertex with respect to the $i$'th facet:
\[ S_{i,j} = b_i - \langle a_i, v_j \rangle \quad \forall i=1,\dots,F, \;\; j=1,\dots,V. \]
Observe that $S_{i,j} \geq 0$ for all $i,j$.

Assuming that the cone $\cK$ is self-dual\footnote{Recall that a cone $\cK \subset \RR^m$ is self-dual if $\cK^{*} = \cK$ where $\cK^*$ is the dual cone defined by $\cK^* = \{y \in \RR^m : y^T x \geq 0 \; \forall x \in \cK \}$. The nonnegative orthant $\RR^m_+$ and the cone of symmetric positive semidefinite matrices $\S^m_+$ are self-dual cones.} (in this paper we will be dealing only with self-dual cones), we say that $S$ admits a $\cK$-factorization \cite{gouveia2011lifts} if there exist vectors $x_1,\dots,x_F \in \cK$ and $y_1,\dots,y_V \in \cK$ such that
\[ S_{i,j} = \langle x_i , y_j \rangle \quad \forall i=1,\dots,F, \;\; j=1,\dots,V. \]

We now recall the definitions of nonnegative rank and psd rank of a nonnegative matrix:
\begin{defn}[Nonnegative rank] The \emph{nonnegative rank} of a nonnegative matrix $S \in \RR^{m\times n}_+$, denoted $\rank_+ S$ is the smallest integer $r$ such that we can write for any $i,j$, $S_{i,j} = \langle x_i, y_j \rangle$ where $x_i, y_j \in \RR^r_+$. Note that $\rank_+ S$ is the smallest $r$ such that $S$ admits a $\RR^r_+$-factorization.
\end{defn}
\begin{defn}[PSD rank, cf. \cite{gouveia2011lifts}]
\label{def:psdrank} The \emph{psd rank} of a nonnegative matrix $S\in\RR^{m\times n}_+$, denoted $\rankpsd S$, is the smallest $r$ such that we can write for any $i,j$, $S_{i,j} = \langle X_i, Y_j \rangle$ where $X_i, Y_j \in \S^r_+$.  Note that $\rankpsd S$ is the smallest $r$ such that $S$ admits a $\S^r_+$-factorization.
\end{defn}

The following key theorem from \cite{gouveia2011lifts} gives a necessary and sufficient condition for the existence of a $\cK$-lift of $P$ in terms of $\cK$-factorization of its slack matrix (the statement of the theorem uses the notion of \emph{nice} cones which is defined in \cite{pataki2012connection}---in this paper we deal with Cartesian products of positive semidefinite cones which are indeed nice cones, as shown in the previously cited paper):

\begin{thm}[\cite{gouveia2011lifts}]
\label{thm:factorization-polytope}
 Let $P \subset \RR^n$ be a polytope and let $\cK$ be a \emph{nice} self-dual cone in $\RR^m$. Then $P$ admits a $\cK$-lift if, and only if, the slack matrix of $P$ admits a $\cK$-factorization.
\end{thm}

In this paper we are interested in the case where $\cK$ is the Cartesian product of the cones $\S^{d}_+$ where $d \geq 1$ is a fixed constant. Define $\xc_{\S^{d}_+}(P)$ to be the least integer $r$ such that $P$ admits a $(\S^{d}_+)^r$-lift.
Also given a nonnegative matrix $A$, define $\rank_{\S^{d}_+}(A)$ to be the least $r$ such that $A$ can be written as the sum of $r$ matrices $A_1,\dots,A_r$ where each $A_i$ admits a $\S^{d}_+$-factorization (i.e., $\rankpsd(A_i) \leq d$).
A consequence of Theorem \ref{thm:factorization-polytope} is that
\[ \xc_{\S^{d}_+}(P) = \rank_{\S^{d}_+}(S) \]
where $S$ is the slack matrix of $P$. Again, observe that for $d=1$, the quantities above are respectively the LP extension complexity and the nonnegative rank.

Our approach to prove Theorem \ref{thm:main} is to show that $\rank_{\S^d_+}(\UDISJ(n))$ is exponentially large in $n$, where $\UDISJ(n)$ is the \emph{unique-disjointness matrix} studied in \cite{fiorini2012linear} and which is a submatrix of the slack matrix of the correlation polytope. The definition of the unique disjointness matrix is given below:

\begin{defn}[Unique disjointness matrix, cf. \cite{fiorini2012linear}]
\label{def:udisj}
The unique disjointness matrix denoted $\UDISJ{(n)}$ is a $2^n\times 2^n$ matrix where rows and columns are indexed by $n$-bit strings and is defined by:
\[ \UDISJ{(n)}_{a,b} = (1-a^T b)^2  \quad \forall a\in \{0,1\}^n, b\in \{0,1\}^n, \]
where the inner product $a^T b$ is understood over $\RR$.
\end{defn}
One can verify that $\UDISJ{(n)}$ is indeed a submatrix of the slack matrix of $\COR(n)$, since for any $a \in \{0,1\}^n$, the following inequality is valid for any $x \in \COR(n)$:
\[ \langle 2\diag(a) - aa^T, x \rangle \leq 1. \]
Furthermore, the slack at vertex $x=bb^T$ where $b \in \{0,1\}^n$ is precisely $(1-a^T b)^2$.

\paragraph{Some notations and terminology} In this paper we will be dealing mostly with matrices where rows and columns are indexed by bit strings in $\{0,1\}^n$, and we assume that they are ordered lexicographically (e.g., if $n=2$, the first row corresponds to $00$, the second row to $01$, the third row to $10$ and the last row to $11$). We will denote the entries of a matrix $M$ using either the subscript notation $M_{a,b}$ or the bracket notation $M[a,b]$ whichever is more convenient. A \emph{disjoint pair} $(a,b)$ is a pair of bit strings such that $a^T b = \sum_{i=1}^n a_i b_i = 0$ (where the summation is understood over $\RR$). In this paper, a \emph{rectangle} $R$ is a 0/1 matrix whose support has the form $I\times J$ where $I$ is a subset of the rows and $J$ is a subset of the columns. We will also sometimes refer to a rectangle as the set $I\times J$ itself, rather than the 0/1 matrix. Finally we will denote by $\1\{T\}$ the indicator function of $T$ which evaluates to $1$ when $T$ is true and $0$ otherwise.

\subsection{Review of the induction argument of Kaibel and Weltge \cite{kaibel2013short}}
\label{sec:reviewproofkaibel}
In \cite{kaibel2013short}, Kaibel and Weltge showed using elementary techniques that\footnote{In fact the lower bound they showed is on the Boolean rank rather than the nonnegative rank, but in this paper we are only interested in $\rank_+$ and $\rank_{\S^d_+}$.} $\rank_+(\UDISJ{(n)}) \geq (3/2)^n$. In this section we briefly review the main idea of their induction argument which will be useful to prove our main result. Let us first define $\cA(n)$ to be the set of rank-one matrices $M \in \RR^{2^n\times 2^n}$ where $M_{a,b} = 0$ whenever $a$ and $b$ intersect in exactly one location:
\begin{equation}
\label{eq:atomsLP}
 \cA(n) = \{ M \in \RR^{2^n \times 2^n}_+ \; : \; M \text{ is rank-one and } M_{a,b} = 0 \; \text{ whenever } a^T b = 1 \}.
\end{equation}
Note that in any decomposition of $\UDISJ{(n)}$ into rank-1 factors, each factor must belong to $\cA(n)$ since $\UDISJ{(n)}_{a,b} = 0$ when $a^T b=1$; this is why we use the notation $\cA(n)$, where $\cA$ stands for ``atoms''.

Also given a matrix $M \in \RR^{2^n \times 2^n}$ denote by $\val(M)$ the number of disjoint pairs $(a,b)$ for which $M_{a,b} > 0$, i.e.,
\[ \val(M) = \left|\left\{(a,b) \in \{0,1\}^n \times \{0,1\}^n \; : \; a^T b = 0 \text{ and } M_{a,b} > 0\right\}\right|. \]
(We will sometimes use the notation $\val_n(M)$ where the subscript $n$ indicates that the matrix $M$ has size $2^n\times 2^n$).
The main idea of the proof of \cite{kaibel2013short} is to show that any nonnegative rank-1 matrix $M \in \cA(n)$ must satisfy $\val(M) \leq 2^n$. Then, since $\val(\UDISJ(n)) = 3^n$ this shows that
\[ \rank_+(\UDISJ(n)) \geq \frac{3^n}{2^n} \]
since in any nonnegative factorization of $\UDISJ(n)$, the rank-one factors must all belong to $\cA(n)$.

To prove that $\val(M) \leq 2^n$ for any $M \in \cA(n)$, one proceeds by induction, as follows: Let $M \in \cA(n)$, and consider the following block-decomposition of $M$:
\[
M = \begin{bmatrix} M^{0,0} & M^{0,1}\\ M^{1,0} & M^{1,1} \end{bmatrix},
\]
where each block $M^{x,y}$ has size $2^{n-1} \times 2^{n-1}$. In this decomposition the top $2^{n-1}$ rows of $M$ correspond to the bit strings that start with a 0 and the bottom $2^{n-1}$ rows are those that start with 1, and similarly for the columns; more formally we have $(M^{x,y})_{a,b} = M[x\concat a,y\concat b]$ where $\concat$ denotes concatenation. The crucial observation in \cite{kaibel2013short} is to note that:
\begin{equation}
 \label{eq:induction-kaibel}
 \val(M) \; \leq \; \val(M^{0,0} + M^{0,1}) \; + \; \val(M^{0,0} + M^{1,0}).
\end{equation}
To see why this is true, observe that we can write:
\[ 
\begin{aligned}
\val(M) \quad &\overset{(*)}{=} \quad \val(M^{0,0}) + \val(M^{0,1}) + \val(M^{1,0})\\
       &= \quad \sum_{\substack{(a,b) \in (\{0,1\}^{n-1})^2\\\text{s.t. } a^T b = 0}} \1\{M^{0,0}_{a,b} > 0\} \; +  \; \1\{M^{0,1}_{a,b} > 0\} \; + \; \1\{M^{1,0}_{a,b} > 0\}.
\end{aligned} \]
where the equality (*) is because in the lower-right block of $M$ (corresponding to $M^{1,1}$) all the bit strings have intersection at least 1---i.e., they are not disjoint.
Now note that for any $(a,b)$ in the summation above, the value of
\[ \1\{M^{0,0}_{a,b} > 0\} \; +  \; \1\{M^{0,1}_{a,b} > 0\} \; + \; \1\{M^{1,0}_{a,b} > 0\} \]
 is at most 2. To see why note that the $2\times 2$ matrix:
\begin{equation}
 \label{eq:2x2matrix}
 \begin{bmatrix} M^{0,0}_{a,b} & M^{0,1}_{a,b}\\ M^{1,0}_{a,b} & M^{1,1}_{a,b} \end{bmatrix}
\end{equation}
is rank-one and $M^{1,1}_{a,b} = 0$: indeed it is rank-one because it is a submatrix of $M$; also $M^{1,1}_{a,b} = 0$ because by definition of $M^{1,1}$ we have $M^{1,1}_{a,b} = M[1\cdot a,1\cdot b]$ and $1\cdot a$ and $1\cdot b$ intersect in exactly one location since $a$ and $b$ are disjoint. Since the matrix \eqref{eq:2x2matrix} is rank-one and $M^{1,1}_{a,b} = 0$, it is not hard to see that we must have either $M^{0,1}_{a,b} = 0$ or $M^{1,0}_{a,b} = 0$. In fact one can verify that the following inequality is true:
\begin{equation}
\label{eq:ineq2x2}
\1\{M^{0,0}_{a,b} > 0\} \; + \; \1\{M^{0,1}_{a,b} > 0\} \; + \; \1\{M^{1,0}_{a,b} > 0\} \;\; \leq \;\; \1\{M^{0,0}_{a,b} + M^{0,1}_{a,b} > 0\} \; + \; \1\{M^{0,0}_{a,b} + M^{1,0}_{a,b} > 0\}.
\end{equation}
Now if we sum inequality \eqref{eq:ineq2x2} over all pairs $(a,b) \in \{0,1\}^{n-1}\times \{0,1\}^{n-1}$ such that $a^T b = 0$ we obtain inequality \eqref{eq:induction-kaibel}.

It now remains to use the induction hypothesis on inequality \eqref{eq:induction-kaibel} to arrive to the result. To use the induction hypothesis we use the following additional crucial fact which one can easily verify: the two matrices $M^{0,0} + M^{0,1}$ and $M^{0,0} + M^{1,0}$ are both elements of $\cA(n-1)$.  Thus by the induction hypothesis $\val(M^{0,0} + M^{0,1}) \leq 2^{n-1}$ and $\val(M^{0,0} + M^{1,0}) \leq 2^{n-1}$ and we get the desired inequality $\val(M) \leq 2^n$.

\subsection{The uniform covering property and an induction lemma}
\label{sec:induction}

In order to prove a lower bound on $\rank_{\S^{d}_+}(\UDISJ(n))$ we use an induction argument inspired by the one of \cite{kaibel2013short} presented above. Our strategy will be to show that any matrix $M$ that admits a $\S^{d}_+$-factorization and such that $M_{a,b} = 0$ when $a^Tb = 1$ must satisfy:
\[ \val(M) \leq t(d)^n \]
where $t(d)$ is a constant satisfying $t(d) < 3$. Then, since $\val(\UDISJ(n)) = 3^n$ this will prove an exponential lower bound on $\rank_{\S^{d}_+}(\UDISJ(n))$:
\[ \rank_{\S^{d}_+}(\UDISJ(n)) \geq \frac{3^n}{t(d)^n} \]
where $3/t(d) > 1$.

The key idea in the induction argument presented in the previous section was to introduce the two matrices $M^{0,0} + M^{0,1}$ and $M^{0,0} + M^{1,0}$  which were chosen so that inequality \eqref{eq:ineq2x2} holds under the additional constraint that they satisfy the induction hypothesis (i.e., they belong to $\cA(n-1)$). In this section we generalize this key idea and we give a general way of constructing such matrices, which will be crucial for us to prove lower bounds.

We start by generalizing the notation $\cA(n)$ to allow for matrices with $K$-factorizations where $K$ is an arbitrary convex cone:
\begin{equation}
\label{eq:atomsK}
 \cA_{K}(n) = \Bigl\{ M \in \RR^{2^n \times 2^n}_+ \; : \; \text{ $M$ admits a $K$-factorization and } M_{a,b} = 0 \; \text{ whenever } a^T b = 1  \Bigr\}.
\end{equation}
Note that the set $\cA(n)$ defined previously in \eqref{eq:atomsLP} corresponds to $\cA_K(n)$ with $K=\RR_+$. When studying $(\S^{d}_+)^r$-factorizations of $\UDISJ(n)$ the set of atoms of interest is $\cA_{K}(n)$ with $K=\S^d_+$. Let also:
\[ \rho_{K}(n) = \max \; \left\{ \; \val(M) \; : \; M \in \cA_{K}(n) \; \right\}. \]
Recall that we are interested in proving upper bounds of the type $\rho_{\S^d_+}(n) \leq t(d)^n$ where $t(d) < 3$.


In this section we prove a general result which allows to obtain an upper bound on $\rho_{K}(n)$ using induction, by simply studying the set of atoms $\cA_{K}(d)$ for a certain \emph{base case}\footnote{Note that this base case $d$ is not related to the size of the semidefinite cones $\S^d_+$; in fact the results in this subsection are general and are valid for any convex cone $K$. Later when we consider the case $K=\S^d_+$ the base case we will study will be the same as the size of the psd cone, and that is why we use the same notation.} $d$. More precisely, we show that if the set $\cA_{K}(d)$, for some \emph{fixed} $d$, has a \emph{$k$-uniform-covering} (cf. definition below) then \emph{for any} $n \geq d$ it holds that:
\begin{equation}
 \label{eq:upperbound-rhoK}
 \rho_K(n) \leq k^{\lfloor (n-1)/d \rfloor +1}.
\end{equation}
We now give the definition of a \emph{$k$-uniform-covering} and we then illustrate it with an example:
\begin{defn}[{\bf $k$-uniform-covering}]
Let $S$ be a subset of nonnegative matrices of size $2^d \times 2^d$. We say that $S$ has a \emph{$k$-uniform-covering} if there exist $k$ rectangles $R_1,\dots,R_k \in \RR^{2^d \times 2^d}$ all of them supported on the disjoint pairs of $\{0,1\}^d\times \{0,1\}^d$ such that the following is true: For any $M \in S$ there exists a one-to-one mapping $\phi$ (that depends on $M$) which maps each nonzero disjoint pair entry of $M$ to a rectangle in $\{1,\dots,k\}$ which is nonzero on this pair. More formally the map $\phi$ must satisfy:
\[ \phi : \{ (x,y) : x^T y = 0 \text{ and } M_{x,y} > 0 \} \rightarrow \{1,\dots,k\} \]
is one-to-one and we have $(x,y) \in R_{\phi(x,y)}$ for all $x,y$ such that $x^T y = 0$ and $M_{x,y} > 0$.
\end{defn}
We now look at a simple example to illustrate the definition of the $k$-uniform covering property.
\begin{example}
\label{ex:uniformcovering}
Recall the set $\cA(n)$ from \eqref{eq:atomsLP} defined by:
\[  \cA(n) = \{ M \in \RR^{2^n \times 2^n}_+ \; : \; M \text{ is rank-one and } M_{a,b} = 0 \; \text{ whenever } a^T b = 1 \}. \]
Consider the set $S = \cA(1)$ which consists of $2\times 2$ rank-one nonnegative matrices $M$ where $M_{1,1} = 0$ (here $M_{1,1}$ is the bottom-right entry of $M$; the top-left entry is $M_{0,0}$). We show in this example that the set $\cA(1)$ has a \emph{2-uniform-covering}. Note that any matrix in $M \in \cA(1)$ must have either $M_{1,0} = 0$ or $M_{0,1} = 0$, i.e., it has one of the two following sparsity patterns:
\[ \begin{bmatrix} \times & 0\\ \times & 0 \end{bmatrix} \quad \text{ or } \quad \begin{bmatrix} \times & \times \\ 0 & 0 \end{bmatrix} \]
where $\times$ indicates a nonnegative entry.
 Consider the following two rectangles $R_1 = \{0\} \times \{0,1\}$ and $R_2 = \{0,1\}\times \{1\}$ which are depicted in Figure \ref{fig:figure_example_rectangles_2x2}.
\begin{figure}[ht]
  \centering
  \includegraphics[width=6cm]{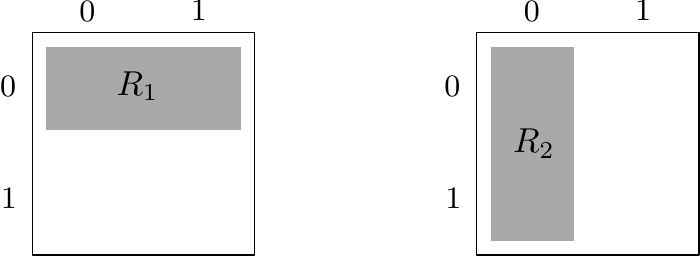}
  \caption{Rectangles $R_1$ and $R_2$ showing that the set $\cA(1)$ has the 2-covering property.}
  \label{fig:figure_example_rectangles_2x2}
\end{figure}

Now let $M$ be any matrix in $\cA(1)$, i.e., $M$ is a rank-one $2\times 2$ matrix such that $M_{1,1} = 0$. We need to construct a map $\phi$ which maps each nonzero entry of $M$ to an associated rectangle which is nonzero on this entry; also each entry has to be mapped to a different rectangle. Since we know that one of the entries $M_{0,1}$ or $M_{1,0}$ must be zero, we can construct the map $\phi$ as follows:
\begin{itemize}
\item If $M_{0,1} = 0$, let $\phi(0,0) = 1$ and $\phi(1,0) = 2$.
\item If $M_{1,0} = 0$, let $\phi(0,0) = 2$ and $\phi(0,1) = 1$.
\end{itemize}
It is easy to see that $\phi$ is a valid map for the definition of 2-uniform-covering property, and thus it shows that $\cA(1)$ has the $2$-uniform-covering property. As a consequence, inequality \eqref{eq:upperbound-rhoK} (proved in Theorem \ref{cor:exponential-upper-bound} below) shows that $\val(M) \leq 2^n$ for all $M \in \cA(n)$ and all $n \geq 1$, and we thus recover the main result of Kaibel and Weltge \cite{kaibel2013short}.

Note that in this example, the key to showing that $\cA(1)$ has the $2$-uniform-covering property was the fact that the sparsity pattern of any $M \in \cA(1)$ must have either $M_{0,1} = 0$ or $M_{1,0} = 0$. In the next section we will prove the $k$-uniform-covering property for other sets of atoms --namely for the set $\cA_{\S^d_+}(d)$-- and to do this we will need a key lemma on the sparsity pattern of the matrices in $\cA_{\S^d_+}(d)$.
\end{example}

We now formally state the main result of this section:
\begin{thm}
\label{cor:exponential-upper-bound}
Let $K$ be a convex cone and assume that, for some $d \in \NN$, $\cA_K(d)$ has a $k$-uniform-covering. Then for any $n\geq d$ and any $M \in \cA_K(n)$ it holds $\val(M) \leq k^{\lfloor (n-1)/d \rfloor +1}$.
\end{thm}
To prove Theorem \ref{cor:exponential-upper-bound} we need the following lemma which can be seen as the generalization of inequality \eqref{eq:induction-kaibel} from Section \ref{sec:reviewproofkaibel} (the matrices $M_1,\dots,M_k$ defined in the lemma below play the same role as the matrices $M^{0,0}+M^{0,1}$ and $M^{0,0}+M^{1,0}$ there).

\begin{lem}
\label{lem:main-induction}
Let $K$ be a convex cone, $d \in \NN$ and assume that $\cA_K(d)$ has a $k$-uniform-covering for some $k\in\NN$. Let $n \geq d$ and let $\bM$ be any matrix in $\cA_{K}(n)$. Consider the $2^d \times 2^d$ block-decomposition of $\bM$ where each block $M^{x,y} \in \RR^{2^{n-d}\times 2^{n-d}}$ is specified by $(x,y) \in \{0,1\}^d \times \{0,1\}^d$ and is defined by:
\[ M^{x,y}[a,b] = \bM[x\concat a, y \concat b] \quad \forall (a,b) \in \{0,1\}^{n-d} \times \{0,1\}^{n-d} \]
where $\concat$ indicates concatenation of bit strings.

\noindent Then we have:
\begin{equation}
 \label{eq:main-ineq}
 \val_n(\bM) \leq \sum_{i=1}^{k} \val_{n-d}(M_i)
\end{equation}
where for each $i=1,\dots,k$, $M_i \in \cA_{K}(n-d)$ is defined by:
\begin{equation}
 \label{eq:defMi}
 M_i = \sum_{(x,y) \in R_i} M^{x,y}
\end{equation}
where the $R_i$'s are the rectangles from the $k$-uniform-covering assumption on $\cA_K(d)$.
\end{lem}

\begin{proof}[Proof of Lemma \ref{lem:main-induction}]
Let $\bM \in \cA_K(n)$ with $n \geq d$, i.e., $\bM$ admits a $K$-factorization and $\bM_{a,b} = 0$ whenever $a^T b = 1$. To prove \eqref{eq:main-ineq}, note that:
\[ \val_n(\bM) = \sum_{\substack{(a,b) \in (\{0,1\}^{n-d})^2 \\ \text{ s.t. } a^T b = 0}}  \val_d\left( (M^{x,y}[a,b])_{x,y} \right) \]
where $(M^{x,y}[a,b])_{x,y}$ denotes the $2^d \times 2^d$ submatrix of $\bM$ whose rows and columns are indexed by $x$ and $y$ respectively (and $(a,b)$ is fixed). We will show that
\begin{equation}
 \label{eq:ineq2}
 \val_d((M^{x,y}[a,b])_{x,y}) \quad \leq \quad \sum_{i=1}^k \mathbf{1} \{ M_i[a,b] > 0 \}.
\end{equation}
where the $M_i$'s are defined according to \eqref{eq:defMi}.
First, note that the $2^d \times 2^d$ matrix $(M^{x,y}[a,b])_{x,y}$ belongs to $\cA_K(d)$: indeed it admits a $K$-factorization because it is a submatrix of $\bM$; also since $a^T b = 0$ we have $M^{x,y}[a,b] = 0$ whenever $x^T y = 1$. Thus by the $k$-uniform-covering assumption on $\cA_K(d)$ there exists a map $\phi$ which injectively maps each nonzero disjoint pair $(x,y)$ of $(M^{x,y}[a,b])_{x,y}$ to a rectangle $R_{\phi(x,y)}$. Now observe that if $M^{x,y}[a,b] > 0$ where $x^T y = 0$ then necessarily $M_{\phi(x,y)}[a,b] > 0$ where $\phi(x,y)$ is the index of the rectangle associated to the $(x,y)$ pair (indeed: since $(x,y) \in R_{\phi(x,y)}$ we have $M_{\phi(x,y)}[a,b] \geq M^{x,y}[a,b] > 0$). Since the map $\phi$ is one-to-one, this shows that any disjoint pair $(x,y)$ where $M^{x,y}[a,b] > 0$ which is counted in the left-hand side of \eqref{eq:ineq2} will also be counted in the right-hand side. Thus this proves inequality \eqref{eq:ineq2}. To obtain \eqref{eq:main-ineq} we simply have to sum \eqref{eq:ineq2} over all $(a,b) \in (\{0,1\}^{n-d})^2$ such that $a^T b = 0$.\\

To finish the proof we have to show that each $M_i$, $i=1,\dots,k$ is an element of $\cA_K(n-d)$. It is not difficult to see that $M_i$ admits a $K$-factorization since $\bM$ has a $K$-factorization and $R_i$ is a rectangle (note: this is the main reason why we require the $R_i$'s to be rectangles). Also if $(a,b) \in (\{0,1\}^{n-d})^2$ is such that $a^T b = 1$ then
\[ M_i[a,b] = \sum_{(x,y) \in R_i} M^{x,y}[a,b] = \sum_{(x,y) \in R_i} \bM[x\concat a, y \concat b] = 0 \]
because for any $(x,y) \in R_i$ we have $(x \concat a)^T (y \concat b) = a^T b = 1$ since $x^T y = 0$.

Thus this finishes the proof of the lemma.
\end{proof}

\begin{proof}[Proof of Theorem \ref{cor:exponential-upper-bound}]
The previous lemma shows that for any $n\geq d$ we have:
\[ \rho_{K}(n) \leq k \cdot \rho_{K}(n-d). \]
Thus, solving the recursion we get (using the fact that $\rho_{K}(d) \leq k$):
\[ \rho_{K}(n) \leq k^{\lfloor (n-1)/d \rfloor + 1}. \]
\end{proof}

%

\paragraph{Positive-semidefinite cones} In this paper, we are interested in the case where $K = \S^{d}_+$. In the next section we will first consider the special case $d=2$ and we will show that $\cA_{\S^2_+}(2)$  has a $k$-uniform-covering with $k=7$. This will be done by analyzing the sparsity pattern of matrices in $\cA_{\S^2_+}(2)$ (cf. Lemma \ref{lem:sparsityell=2}).
In Section \ref{sec:uniformcovering-general} we will prove the general case and we will show that $\cA_{\S^{d}_+}(d)$ has a $k$-uniform-covering with $k = 3^d - 1 < 3^d$. The proof is by induction and it also relies on a key lemma analyzing the sparsity pattern of matrices in $\cA_{\S^d_+}(d)$ (cf. Lemma \ref{lem:sparsity}).

\subsection{A uniform covering for the case $d=2$}
\label{sec:uniformcovering-d=2}
The objective of this section is to show the following result:
\begin{thm}
\label{thm:uniformcovering-d=2}
$\cA_{\S^2_+}(2)$ has a $k$-uniform-covering with $k=7$.
\end{thm}
If we combine the theorem above with Theorem \ref{cor:exponential-upper-bound} from Section \ref{sec:induction} we get the following exponential lower bound on the size of $\S^2_+$-lifts of $\COR(n)$:
\[ \rank_{\S^2_+}(\UDISJ(n)) \geq \frac{1}{\sqrt{7}} \left(\sqrt{\frac{9}{7}}\right)^n \approx 0.37 \times 1.13^n. \]


To prove Theorem \ref{thm:uniformcovering-d=2} we proceed by analyzing the possible sparsity patterns of matrices $M \in \cA_{\S^2_+}(2)$, i.e., $4\times 4$ matrices with  $\rankpsd(M) \leq 2$ and $M_{a,b} = 0$ for $a^T b = 1$.
\begin{lem}
\label{lem:sparsityell=2}
Any $4\times 4$ matrix $M \in \cA_{\S^2_+}(2)$ has one of the following six sparsity patterns below:
{\small
\[
\begin{array}{rl}
& 
\textup{(1)} \;\; 
\begin{bmatrix}
\times & \times & \times & \times\\
\times & 0 & 0 & 0\\
\times & 0 & 0 & 0\\
\times & 0 & 0 & ?
\end{bmatrix}
\quad \textup{or} \quad
\textup{(2)} \;\;
\begin{bmatrix}
\times & \times & \times & 0\\
\times & 0 & \times & 0\\
\times & \times & 0 & 0\\
0 & 0 & 0 & ?
\end{bmatrix}
\quad \textup{or} \quad
\textup{(3)} \;\;
\begin{bmatrix}
\times & \times & \times & \times\\
0 & 0 & 0 & 0\\
\times & \times & 0 & 0\\
\times & 0 & 0 & ?
\end{bmatrix}
\\
& 
\\
\textup{or}
&
\textup{(4)} \;\;
\begin{bmatrix}
\times & \times & \times & \times\\
\times & 0 & \times & 0\\
0 & 0 & 0 & 0\\
\times & 0 & 0 & ?
\end{bmatrix}
\quad \textup{or} \quad
\textup{(5)} \;\;
\begin{bmatrix}
\times & 0 & \times & \times\\
\times & 0 & \times & 0\\
\times & 0 & 0 & 0\\
\times & 0 & 0 & ?
\end{bmatrix}
\quad \textup{or} \quad
\textup{(6)} \;\;
\begin{bmatrix}
\times & \times & 0 & \times\\
\times & 0 & 0 & 0\\
\times & \times & 0 & 0\\
\times & 0 & 0 & ?
\end{bmatrix}
\end{array}
\]
}
In particular $\val(M) \leq 7$. 
\end{lem}
\begin{proof}
Let $M \in \cA_{\S^2_+}(2)$ and let $M_{a,b} = \langle U_a, V_b \rangle$ be a psd-rank-2 factorization of $M$, where $U_a, V_b \in \S^2_+$. Note that, by ``default'', the $(a,b)$ entry of $M$ is zero whenever $a^T b = 1$. Thus $M$ has the following sparsity pattern:
{\small
\[
\begin{array}{cc}
& \begin{array}{rrrr} 00 & 01 & 10 & 11 \end{array}\\
\begin{array}{r} 00 \\ 01 \\ 10 \\ 11 \end{array} &
\left[
\begin{array}{rrrr}
\times & \times & \times & \times\\
\times & 0 & \times & 0\\
\times & \times & 0 & 0\\
\times & 0 & 0 & ?
\end{array}
\right]
\end{array}
\]}
Using the assumption that $\rankpsd(M) \leq 2$ we need to show that $M$ has some more zeros in some specific locations. For $a\in\{0,1\}^2 , b \in \{0,1\}^2$ let $\cU_a = \Im(U_a)$ and $\cV_b = \Im(V_b)$. We distinguish the following cases:
\begin{itemize}
\item If either $\cU_{01} = \cU_{10}$ or $\cV_{01} = \cV_{10}$ then necessarily $M$ has the pattern (1), i.e., $M_{01,10} = M_{10,01} = 0$. Indeed, assume for example that $\cV_{01} = \cV_{10}$. Since $M_{01,01} = 0$ we have that $U_{01} V_{01} = 0$, i.e., $\Im(V_{01}) \subseteq \Ker(U_{01})$. Thus $\Im(V_{10}) = \Im(V_{01}) \subseteq \Ker(U_{01})$, hence $U_{01} V_{10} = 0$ and $M_{01,10} = 0$. We use the same reasoning to show that $M_{10,01} = 0$.
\item Otherwise we necessarily have $\cU_{10} \neq \cU_{01}$ and $\cV_{01} \neq \cV_{10}$. We distinguish the following subcases:
\begin{itemize}
\item If $\cU_{10} + \cU_{01} = \RR^2$ and $\cV_{10} + \cV_{01} = \RR^2$, then necessarily $\cU_{11} = \cV_{11} = \{0\}$ and thus $M_{11,00} = M_{00,11} = 0$ and hence we are in pattern (2).
\item If $\cU_{10} + \cU_{01} \subsetneq \RR^2$: In this case, since $\cU_{10} \neq \cU_{01}$ we have necessarily either $\cU_{10} = \{0\}$ or $\cU_{01} = \{0\}$. Thus either the 2rd or 3rd row is identically zero and so we are either in pattern (3) or (4).
\item If $\cV_{10} + \cV_{01} \subsetneq \RR^2$: Using the same reasoning as in the previous case, then either the 2rd or 3rd column of $M$ is zero and so we are either in pattern (5) or (6).
\end{itemize}
\end{itemize}
\end{proof}

We now exhibit 7 rectangles and we show that $\cA_{\S^2_+}(2)$ has a 7-uniform-covering. Consider the following four rectangles called $\ra,\rb,\rc,\rd$, which are supported on the disjoint pairs:
\[
\begin{aligned}
 \ra &= \{00\} \times \{00,01,10,11\}\\
 \rb &= \{00,01,10,11\}\times \{00\}\\
 \rc &= \{00,01\}\times \{00,10\}\\
 \rd &= \{00,10\}\times \{00,01\}
\end{aligned}
\]
We choose the seven rectangles $R_1,\dots,R_7$ as follows: we take one copy of $\ra$ and two copies of the three other rectangles, i.e.,
\[ R_1 = \ra, \quad R_2 = \rb_1, \; R_3 = \rb_2, \quad R_4 = \rc_1, \; R_5 = \rc_2, \quad R_6 = \rd_1, \; R_7 = \rd_2. \]
where we used subscripts (e.g., $\rb_1,\rb_2$) to indicate two copies of the same rectangle.
To show that these rectangles are valid, i.e., they satisfy the requirement of a uniform-covering we provide the map $\phi$ for the 6 possible sparsity patterns of a matrix $M \in \cA_{\S^2_+}(2)$. These are shown in Figure \ref{fig:map_phi_ell=2} below -- cf. caption of the figure for details. This terminates the proof of Theorem \ref{thm:uniformcovering-d=2} and shows that $\cA_{\S^2_+}(2)$ has a 7-uniform-covering.
\begin{figure}[ht]
{\small
\[
\begin{array}{l}
\textup{(1)} \;\; 
 \begin{bmatrix}
\rb_2 & \rd_2 & \rc_1 & \ra\\
\rc_2 &      &      & \\
\rd_1 & & &\\
\rb_1
\end{bmatrix}
\quad
\textup{(2)} \; \;
\begin{bmatrix}
\rb_2 & \ra & \rc_1 & \phantom{\ra}\\
\rb_1 &      & \rc_2 & \\
\rd_1 & \rd_2 & &\\
 & & & 
\end{bmatrix}
\quad
\textup{(3)} \; \;
\begin{bmatrix}
\rc_2 & \rd_1 & \rc_1 & \ra\\
 &      &  & \\
\rb_2 & \rd_2 & &\\
\rb_1 & & & 
\end{bmatrix}
\\
 \\
\textup{(4)} \; \;
\begin{bmatrix}
\rd_2 & \rd_1 & \rc_1 & \ra\\
\rb_2 &     & \rc_2 & \\
    &     & &\\
\rb_1 & & & 
\end{bmatrix}
\quad \textup{(5)} \; \;
\begin{bmatrix}
\rd_2 &  \phantom{\rd_2}  & \rc_1 & \ra\\
\rb_2 &     & \rc_2 & \\
\rd_1 &     & &\\
\rb_1 &     & & 
\end{bmatrix}
\quad \textup{(6)} \; \;
\begin{bmatrix}
\rc_2 & \rd_1 & \phantom{\rc_1} & \ra\\
\rc_1 &      &  & \\
\rb_2 & \rd_2 & &\\
\rb_1 & & & 
\end{bmatrix} 
\end{array}
\]
}
\caption{To show that the rectangles $R_1,\dots,R_7$ are valid in the sense of uniform-covering, we need to be able to associate to each nonzero disjoint pair of $M$ a rectangle $R_i$ which is nonzero on this entry; furthermore each rectangle can only be used once (the mapping has to be one-to-one). The figure above shows how to construct this mapping for the 6 possible sparsity patterns of matrices in $\cA_{\S^2_+}(2)$.
For example, for the first sparsity pattern we associate the $M_{00,00}$ entry to rectangle $\rb_2$, the $M_{00,01}$ entry to rectangle $\rd_2$, etc. (here $\rb_1$ refers to the first copy of rectangle $\rb$, and $\rb_2$ refers to the second copy of rectangle $\rb$). The disjoint entries that are blank need not be associated to any rectangle because they are equal to zero (cf. the sparsity patterns of lemma \ref{lem:sparsityell=2}).}
\label{fig:map_phi_ell=2}
\end{figure}

\subsection{A uniform-covering for the general case}
\label{sec:uniformcovering-general}

In this section we treat the general case where $d$ is abitrary and we prove the following theorem:
\begin{thm} 
\label{thm:main-rectangles}
For any $d \geq 1$, $\cA_{\S^d_+}(d)$ has a $k$-uniform-covering with $k=3^d-1 < 3^d$.
\end{thm}
We need the following lemma whose proof is in Appendix \ref{sec:prooflemmasparsity}:
\begin{lem}
\label{lem:sparsity}
If $M \in \cA_{\S^{d}_+}(d)$ then $M$ has at least one zero entry on the antidiagonal, i.e., there exists $a \in \{0,1\}^{d}$ such that $M_{a,\bar{a}} = 0$ (where $\bar{a} \in \{0,1\}^{d}$ is the bitwise complement of $a$).
\end{lem}
\begin{rem}
One can verify in the special case $d=2$ treated in Lemma \ref{lem:sparsityell=2} that there is indeed one zero on the antidiagonal in all 6 sparsity patterns.
\end{rem}

Define $\C(d)$ to be the set of $2^d \times 2^d$ matrices $M$ such that $M_{a,b} = 0$ whenever $a^T b = 1$ and such that $M$ has at least one zero on the antidiagonal, i.e.,
\begin{equation*}
 \C(d) = \Bigl\{ M \in \RR^{2^d \times 2^d}_+ \; : \; M_{\alpha,\bar{\alpha}} = 0 \; \text{ for some $\alpha \in \{0,1\}^d$} \; \text{ and } M_{a,b} = 0 \; \text{ whenever } a^T b = 1 \;   \Bigr\}.
\end{equation*}
In this section we will show that $\C(d)$ has a $k$-uniform-covering with $k = 3^{d}-1$. This will prove Theorem \ref{thm:main-rectangles} since by the previous lemma, $\cA_{\S^{d}_+}(d) \subseteq \C(d)$. 

The proof of the theorem is by induction on $d$. It is easy to see that $\C(1)$ is true for $d=1$ (with $k = 3^1 - 1 = 2$); indeed this is similar to Example \ref{ex:uniformcovering} from Section \ref{sec:induction}. The following lemma shows how to construct rectangles for the $d$'th level given rectangles for the $d-1$'st level.
\begin{lem}
\label{lem:rectangles_Sd}
Assume $\C(d-1)$ has a $k_{d-1}$-uniform-covering with rectangles  $R_1,\dots,R_{k_{d-1}}$. Then $\C(d)$ has a $k_d$-uniform-covering with $k_d = k_{d-1} + 2\cdot 3^{d-1}$.
\end{lem}
Note that if we solve the recursion $k_d = k_{d-1} + 2\cdot 3^{d-1}$ with the initial value $k_1 = 2$ we get the desired $k_d = 3^d-1$.
\begin{proof}[Proof of lemma \ref{lem:rectangles_Sd}]
Let $R_1,\dots,R_{k_{d-1}}$ be the $k_{d-1}$ rectangles from the assumption.
We are going to define $k_d = k_{d-1} + 2\cdot 3^{d-1}$ new rectangles for level $d$ as follows. Define for each $x,y \in (\{0,1\}^{d-1})^2$ such that $x^T y = 0$ the following rectangles in $ \{0,1\}^d \times \{0,1\}^d$:
\[ A_{xy} = \{0 \concat x\} \times \{0\concat y,1\concat y\}, \]
\[ B_{xy} = \{0\concat x, 1\concat x\} \times \{0\concat y\}, \]
and for $i=1,\dots,k_{d-1}$ define
\[ C_i = \{ (0\concat x, 0\concat y) : (x,y) \in R_{i} \}.\]
\begin{figure}[ht]
  \centering
  \includegraphics[width=12cm]{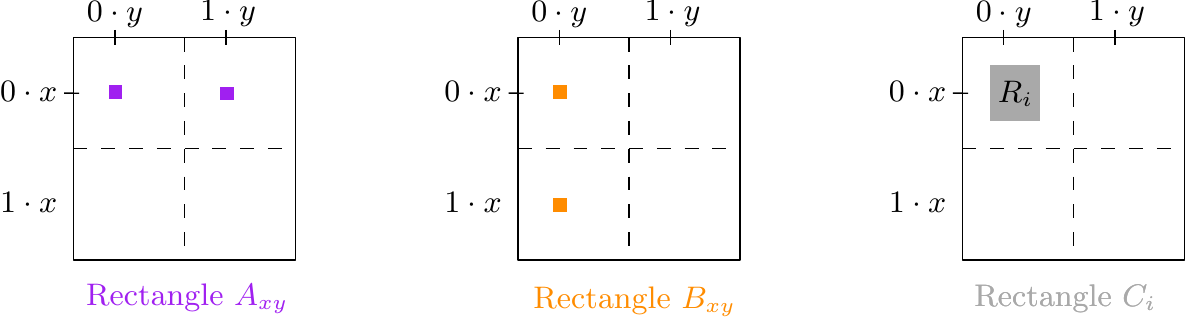}
  \caption{Illustration of the rectangles $A_{xy}$, $B_{xy}$ and $C_i$}
  \label{fig:figure_definition_rectangles}
\end{figure}
 This defines $2 \cdot 3^{d-1} + k_{d-1}$ rectangles which are supported on the disjoint pairs of $\{0,1\}^d$ (see Figure \ref{fig:figure_definition_rectangles} for an illustration of these rectangles---note for example that the number of elements in the rectangles $A_{xy}$ and $B_{xy}$ is equal to two). We claim that these rectangles are valid (in the sense of the uniform-covering property).
Indeed, let $M \in \C(d)$. Let $(a,\bar{a})$ be the element on the antidiagonal for which $M_{a,\bar{a}} = 0$.  We are going to map the nonzero disjoint pairs of $M$ to the rectangles defined above as follows: We proceed in two steps where in the first step we deal with entries that lie either in the top-right or bottom-left block of the matrix, and in the second step we deal with entries in the top-left block:
\begin{enumerate}
\item Let $(x,y)$ be a disjoint pair in $\{0,1\}^d \times \{0,1\}^d$ such that $M_{x,y} > 0$:
\begin{itemize}
\item If $(x,y)$ lies in the top-right block (i.e., if $y_1 = 1$) we use the appropriate rectangle $A_{xy}$.
\item If $(x,y)$ lies in the bottom-left block (i.e., if $x_1 = 1$) we use the appropriate rectangle $B_{xy}$.
\end{itemize}
\item Now it remains to map the nonzero pairs in the top-left block. For this we proceed as follows: Recall that $a \in \{0,1\}^d$ is such that $M[a,\bar{a}] = 0$. We will assume without loss of generality that $(a,\bar{a})$ lies in the top-right block, i.e., $a_1 = 0$ (cf. Figure \ref{fig:figure_step2_rectangles}).
\begin{figure}[ht]
  \centering
  \includegraphics[width=4cm]{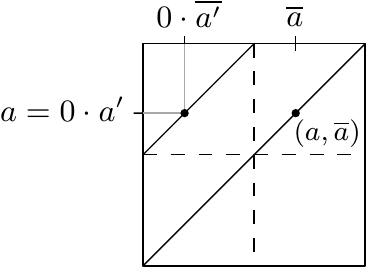}
  \caption{Illustration of step 2 in proof of Lemma \ref{lem:rectangles_Sd}}
  \label{fig:figure_step2_rectangles}
\end{figure}
Let $a'=a_{2\dots d}$. Consider the antidiagonal of the top-left block and consider the element $(0\concat a',0\concat \overline{a'})$ which lies on this antidiagonal. Since $M_{a,\bar{a}} = 0$ the rectangle $A_{a'\overline{a'}}$ was not used in step 1 above and so we can use it to map the entry $(0\concat a',0\concat \overline{a'})$ if it is $>0$. We can now effectively think of this entry as being a zero and we can thus think of the top-left block as a $2^{d-1} \times 2^{d-1}$ matrix in $\C(d-1)$. Thus by the induction hypothesis, we can map all its nonzero disjoint pairs with $k_{d-1}$ rectangles which are the $C_i$'s.
Thus this terminates the proof of the lemma.
\end{enumerate}
\end{proof}

\paragraph{Putting things together}

If we combine Theorem \ref{thm:main-rectangles} with Theorem \ref{cor:exponential-upper-bound}, we get the following result:
\begin{cor}
Let $d \geq 1$ be fixed. For any $n \geq d$ and any $M \in \cA_{\S^d_+}(n)$, it holds $\val(M) \leq (3^d-1)^{\lfloor (n-1)/d \rfloor + 1}$.
\end{cor}
\noindent Thus, using this corollary and the fact that $\val(\UDISJ(n)) = 3^n$, we have that for any any fixed $d$, and any $n\geq d$:
\begin{equation}
\label{eq:lowerbound-generalanalysis-d=2}
\rank_{\S^d_+}(\UDISJ(n)) \geq \frac{3^n}{(3^d-1)^{\lfloor (n-1)/d \rfloor + 1}} \geq \frac{1}{(3^d-1)^{1-1/d}} \left((1-3^{-d})^{-1/d}\right)^n
\end{equation}
which is the lower bound stated in Theorem \ref{thm:main}.

\begin{rem}
Note that for the case $d=2$ the result above gives the lower bound:
\[ \rank_{\S^2_+}(\UDISJ(n)) \geq \frac{1}{\sqrt{8}} \left(\sqrt{\frac{9}{8}}\right)^n. \]
which is slightly weaker than the lower bound we obtained in the previous section by a more refined analysis of the case $d=2$.
\end{rem}

\paragraph{Acknowledgments} Hamza Fawzi would like to thank Omar Fawzi for useful comments on the manuscript.

\appendix

\section{Proof of Lemma \ref{lem:sparsity}}
\label{sec:prooflemmasparsity}

Let $M_{a,b} = \langle U_a, V_b \rangle$ be a psd-factorization of $M$, where $U_a, V_b \in \S^{d}_+$. Let $\1 = 1\dots 1 \in \{0,1\}^{d}$ be the all-ones bit string, and let $e_i \in \{0,1\}^{d}$ be the bit string with a 1 in $i$'th position, and $0$'s elsewhere. Note that by our assumption we have $M_{\1,e_i} = \langle U_{\1}, V_{e_i} \rangle = 0$ for all $i \in \{1,\dots,d\}$ since $\1^T e_i = 1$. Hence this means that $\Im V_{e_i} \subseteq \Ker U_{\1}$ for all $i \in \{1,\dots,d\}$, and thus
\[ \Im V_{e_1} + \dots + \Im V_{e_{d}} \subseteq \Ker U_{\1}. \]
We now distinguish two cases:
\begin{itemize}
\item If $\Im V_{e_1} + \dots + \Im V_{e_{d}} = \RR^{d}$ we are done since then $\Ker U_{\1} = \RR^{d}$, i.e., $U_{\1} = 0$ which implies that $M_{a,\bar{a}} = 0$ with $a = \1$ (in fact the whole row $M_{\1,\cdot}$ is zero).
\item Otherwise assume $\Im V_{e_1} + \dots + \Im V_{e_{d}} \subsetneq \RR^{d}$. If $\Im V_{e_1} = \{0\}$, i.e., $V_{e_1} = 0$ then the whole column $M_{\cdot,e_1}$ is zero and so the claim is true. If $\Im V_{e_1} \neq \{0\}$ one can show that we have necessarily 
\[ \Im V_{e_{p+1}} \subseteq \Im V_{e_1} + \dots + \Im V_{e_p} \]
for some $p$ (indeed, one can consider the nondecreasing sequence of subspaces $F_i = \Im V_{e_1} + \dots + \Im V_{e_i}$ and observe that there must there exist $p$ such that $F_p =  F_{p+1}$).
Define $a = \overline{e_{p+1}}$. We claim that $M_{a,\bar{a}} = 0$. In fact, note that $a^T e_i = 1$ for any $i \neq p$, and thus $M_{a,e_i} = \langle U_a, V_{e_i} \rangle = 0$ for all $i \neq p$. This means that $\Im V_{e_i} \subseteq \Ker U_{a}$ for all $i\neq p$ and thus
\[ \Im V_{e_1} + \dots + \Im V_{e_p} \subseteq \Ker U_a \]
Hence we get $\Im V_{e_{p+1}} \subseteq \Ker U_a$, which means $M_{a,e_{p+1}} = 0$. Since, by definition, $e_{p+1} = \bar{a}$, this terminates the proof.
\end{itemize}

\bibliographystyle{alpha}
\bibliography{../../../bib/nonnegative_rank}

\end{document}